\documentclass[11pt,reqno]{amsart}
\usepackage{amsmath,amsfonts,amsthm,amssymb}
\usepackage[colorlinks]{hyperref}
\usepackage[all]{xy}
\usepackage{mathrsfs}
\usepackage{xcolor}

\RequirePackage[shortlabels]{enumitem}
\setlist{nosep}
	
\title[Chapter 5]{On the uniruled Voisin divisor on the LLSvS variety}
\author{Franco Giovenzana}

\address{Fakult\"at f\"ur Mathematik\\
Technische Universit\"at Chemnitz \\
Reichenhainer Strasse 39 \\
09126 Chemnitz \\
Germany \\
}
\email{franco.giovenzana@mathematik.tu-chemnitz.de}

\newcommand \bL {\mathbb L}
\renewcommand \O {\mathcal O}
\renewcommand \P {\mathbb P}
\newcommand \bZ {\mathbb Z}

\newcommand \lra \longrightarrow
\newcommand \xra \xrightarrow
\newcommand \xla \xleftarrow
\newcommand \x \times

\DeclareMathOperator \Pic {Pic}
\DeclareMathOperator \Hilb {Hilb}
\DeclareMathOperator \Gr {Gr}

\newcommand \sing {\mathrm{sing}}
\newcommand \gtc {\mathrm{gtc}}

\newtheorem{thm}{Theorem}
\newtheorem{prop}[thm]{Proposition}
\newtheorem{cor}[thm]{Corollary}
\newtheorem{lem}[thm]{Lemma}
\newtheorem{defi}[thm]{Definition}

\theoremstyle{definition}

\renewcommand \phi \varphi

\begin{document}

\maketitle
\begin{abstract}
Let $Y$ be a smooth cubic fourfold, $F$ be its Fano variety of lines and $Z$ be its associated LLSvS variety, parametrizing families of twisted cubics and some of their degenerations. In this short note, we show that the divisor of singular cubic surfaces on $Z$ has two irreducible components, one of which coincides with the uniruled branch divisor of a resolution of the Voisin map $F\times F \dashrightarrow Z$.
\end{abstract}
 
\section*{Introduction}

Given a smooth cubic fourfold $Y\subset \P^5$, its Fano variety of lines $F$ is a hyperk\" ahler variety \cite[Proposition~1]{BD}. Assuming further that $Y$ does not contain any plane, one constructs starting from the variety of twisted cubics another hyperk\" ahler variety $Z$ of dimension 8 \cite[Theorem B]{LLSvS}. Their geometry is deeply related as Voisin showed by constructing a degree 6 rational map \cite[Proposition~4.8]{Voisin}
\[
\phi: F\x F \dashrightarrow Z.
\]

Twisted cubics have Hilbert polynomial $3t+1$, their flat degenerations are called generalised twisted cubics.
Generalised twisted cubics on a singular cubic surface of type $A_1$ fall in three different types, which we are going to call $\alpha, \beta$, and $\gamma$. They determine two divisors $D_\alpha, D_\beta$, on $Z$, see Definition \ref{def-divisors}.

\begin{thm}
The divisor $D \subset Z$ of generalised twisted cubics lying on singular cubic surfaces has two irreducible components
\[
D = D_\alpha \cup D_\beta .
\]
The first component coincides with the uniruled branch divisor of a resolution of the Voisin map.
\end{thm}

\subsection*{Acknowledgements}
The content of this paper is part of my PhD thesis, I would like to express my gratitude to Christian Lehn for having introduced me to the problem and for his constant support.
I was supported by DFG research grants Le 3093/2-1 and Le 3093/2-2.

\section*{Basic properties of the LLSvS variety}
\subsection*{Construction of the LLSvS variety}
We recollect basic properties of the variety $Z$ from \cite{LLSvS}. Let $Y\subset \P(V)\simeq\P^5$ be a smooth cubic fourfold not containing any plane and let $M$ be the variety of generalised twisted cubics on $Y$, that is, the irreducible component of $\Hilb^{3t+1}(Y)$ containing smooth twisted cubics.
For any curve $C\in M$ its linear span $ E:= \langle C \rangle $ is a $\P^3$ that cuts $Y$ in the cubic surface $S_C:=E\cap Y$. These associations give rise to the diagram
\[\xymatrix{
M \ar[r] \ar[dr]^\sigma & \P (S^3 \mathcal{W}^*)\ar[d] \\
& \Gr(V,4) \ar@/_/[u]_{s}
}
\]
where $\mathcal{W}$ is the universal quotient bundle over the Grassmannian $\Gr(V,4)$ and the section $s$ is given by cutting $Y$ with a $\P^3$.
The morphism $\sigma$ factors through a smooth irreducible variety  $Z'$ of dimension 8:
\[\xymatrix{
M \ar[rd]^\sigma \ar[d]^a\\
Z' \ar[r]^-b & \Gr(V,4)
}
\]
where $a$ is an \' etale locally trivial $\P^2$-fibration and $b$ is generically finite \cite[Theorem~B]{LLSvS}. It is in effect finite over the open set $U_{ADE}\subset Z'$ of surfaces that are either smooth or with ordinary double points. For the complement we have the following estimate \cite[Corollary~3.11, Proposition~4.2, Proposition~4.3]{LLSvS}:
\begin{align} \label{estimate}
\dim (Z' \setminus U_{ADE}) \leq 6.
\end{align}
The irreducible holomorphically symplectic variety $Z$ is obtained by the contraction $\pi:Z'\to Z$ of the irreducible divisor $D_{nCM}\subset Z'$ of families of curves that are not arithmetically Cohen-Macaulay \cite[Theorem~4.11]{LLSvS}.

\subsection*{The divisor of singular cubic surfaces}

The projection $\P (S^3\mathcal{W}^*) \to \Gr(V,4)$ is equivariant for the natural action of $\mathrm{PGL}(6)$. The locus $D_\sing \subset \P (S^3\mathcal{W}^*)$ of all singular cubic surfaces in $\P(V)$ surjects onto the Grassmannian, the fiber over a point $W_0$ is the divisor $D_{\sing, W_0}\subset \P(S^3W_0^*)$ of singular cubic surfaces in $\P(W_0)$.
 Furthermore, $D_{\sing,W_0}$ is irreducible \cite[Theorem~2.2]{huybrechts-cubic-hypersurfaces} and locally stratified depending on the singularity type of the parametrised surfaces, the $A_1$ locus forming an open set in $D_{\sing,W_0}$.
 Since $D_\sing$ coincides with the orbit of $D_{\sing, W_0}$ under the action of $\mathrm{PGL}(6)$ we conclude that $D_\sing$ is an irreducible divisor and the $A_1$-locus is open in it.
Pulling back this divisor along $s\circ b: Z' \to \P (S^3\mathcal{W}^*)$ and taking advantage of both the finiteness of $s\circ b$ on the ADE-locus and the estimate \eqref{estimate} we get the following

\begin{prop}
Let $D' \subset Z'$ be the image under $a$ of the locus of curves lying on singular cubic surfaces. Then $D'$ is a divisor. Moreover, the preimage under $s\circ b$ of singular cubic surfaces of type $A_1$ is a dense open set in $D'$.
\end{prop}
Analogously singular cubic surfaces determine a divisor $D^{\Gr}:= s^{-1}(D_\sing)$ in $\Gr(V,4)$ and a divisor $D:=\pi(D')$ in $Z$.
In light of this proposition, when discussing the irreducible components of $D'$ it will suffice to treat only the $A_1$ locus.

\subsection*{Twisted cubics on $A_1$-singular cubic surfaces}
Whereas twisted cubics on smooth cubic surfaces are a classical subject of study, twisted cubics on surfaces with ordinary double points are well explained in \cite[\S~2.1]{LLSvS}, whose content we briefly recall here and then make explicit in the specific case of $A_1$ singularities. For basic facts on the root lattice $E_6$ in connection with the geometry of the cubic surface we point to \cite[\S 9]{Dolgachev}, \cite{Demazure}, for lines on singular cubic surfaces \cite{bruce-wall}.
 
Given a singular cubic surface $S$ with ordinary double points, its minimal resolution $\tilde{S}$ is a weak del Pezzo surface and the orthogonal complement $K_{\tilde{S}}^\perp \subset \Pic(\tilde{S})$ is a lattice of type $E_6$. 
The exceptional divisor of the resolution $r:\tilde{S}\to S$ consists of $(-2)$-curves, which form a subset of the root system $R:=\{\alpha:\alpha^2 = 0\} \subset K_{\tilde{S}}^\perp$ and generate a subroot system $R_0\subset R$. Let $W(R_0)$ be the Weyl group generated by  reflections of elements in $R_0$, then \cite[Theorem~2.1]{LLSvS} gives a description of the Hilbert scheme with the reduced structure of generalised twisted cubics on $S$:
\begin{align} \label{Hilb(S)}
\Hilb^{\gtc}(S)_{\mathrm{red}} \simeq R/W(R_0) \times \P^2.
\end{align}
For any $\alpha \in R\setminus R_0$ and for any curve $C\in |\alpha - K_{\tilde{S}}|$ the image $r(C)$ is a generalised twisted cubic on $S$. Conversely, the pullback of any aCM-curve on $S$ lies in such a linear system \cite[Proposition~2.2, Proposition~2.5, Proposition~2.6]{LLSvS}. On the other hand, roots in $R_0$ correspond to families of nCM curves.
We now want to make \eqref{Hilb(S)} explicit for $A_1$-singular surfaces.

Let $S\subset \P^3$ be a singular cubic surface of type $A_1$ with singular point $P$, then its minimal resolution $\tilde{S}$ is the blow-up with center $P$, whose exceptional divisor is a $(-2)$-curve $\Gamma$. The linear system $|-\Gamma - K_{\tilde{S}}|$ realises $\tilde{S}$ as a blow-up of $\P^2$ in 6 points $P_1,...,P_6$ lying on a quadric $Q$, which is exactly the image of $\Gamma$. In a picture:
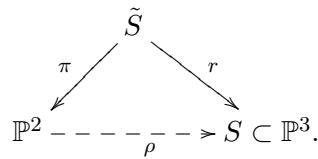
\[\xymatrix{
&\tilde{S}\ar[dl]_\pi\ar[rd]^{r}\\
\P^2 \ar@{-->}[rr]_\rho && S\subset\P^3.
}
\]

Here $\rho$ is the rational map given by the linear system of cubics passing through the 6 points.\\
The Picard group of $\tilde{S}$ is generated by the pullback $H$ of the hyperplane class of $\P^2$ and the 6 exceptional divisors $E_1,...,E_6$. The surface $S$ contains 21 lines: 6 lines are the image $\overline{E}_i$ of the exceptional divisors, they pass through the singular point; moreover, the line trough any two distinct points $P_i,P_j$ is mapped via $\rho$ to the unique line $R_{ij}$ on $S$ intersecting both $\overline{E}_i$ and $\overline{E}_j$ not in $P$.

The canonical bundle $K_{\tilde{S}}$ has class $-3H+E_1+E_2+E_3+E_4+E_5+E_6$, its orthogonal complement $K_{\tilde{S}}^\perp\subset \Pic(\tilde{S})$ is a root lattice of type $E_6$. It has 72 roots:
\begin{align*}
\alpha_{ij} &= E_i - E_j, \textrm{ for $i \not = j$ }\\ 
\pm\beta_{ijk} &= \pm (H-E_i-E_j-E_k),  \textrm{ for $i,j,k$ pairwise distinct};\\
\pm\gamma &= \pm (2H-E_1-E_2-E_3-E_4-E_5-E_6).
\end{align*}
We are going to call the roots $\alpha_{i,j}$ of type $\alpha$, the roots $\pm\beta_{i,j,k}$ of type $\beta$, and the roots $\pm \gamma$ of type $\gamma$. The unique effective root is $\gamma$, the reflection with center $\gamma$ fixes the roots of type $\alpha$, whereas its action on $\beta$-roots is:
\[
\pm\beta_{ijk} \mapsto \mp\beta_{lmn}
\]
with $\{ i,j,k,l,m,n \} = \{ 1,2,3,4,5,6 \}$.
Hence according to \cite[Theorem~2.1]{LLSvS} the Hilbert scheme with the reduced structure is the disjoint union of 51 copies of $\P^2$
\[
\Hilb^{\gtc}(S)_{\mathrm{red}} \simeq \bigsqcup_{\substack{i,j \\  i\not = j}} |\alpha_{ij} - K_{\tilde{S}}|
\ \sqcup \bigsqcup_{\substack{i,j,k \\  \mathrm{pairwise}\\ \mathrm{distinct}}} |\beta_{ijk} - K_{\tilde{S}}|
\ \sqcup \ |-\gamma - K_{\tilde{S}}|.
\]
The curves parametrised by $|-\gamma - K_{\tilde{S}}|$ are not arithmetically Cohen Macaulay and will play no role in the following. For any curve $C$ in $|\alpha_{ij} - K_{\tilde{S}}|$ or in $|\beta_{ijk} - K_{\tilde{S}}|$ the schematic image $r(C)$ is a generalised twisted cubic on $S$.

An ordered pair $(\overline{E}_{i},\overline{E}_{j})$ of distinct lines through the singular point determines the family of generalised twisted cubic on $S$ corresponding to the linear system $|E_i - E_j -  K_{\tilde{S}}|$. In contrast any triple $(\overline{E}_{i},\overline{E}_{j},\overline{E}_{k})$ of pairwise distinct lines through the singular point determines the family of generalised twisted cubic on $S$ corresponding to the linear system $|H - E_i - E_j - E_k -  K_{\tilde{S}}|$.

We recapitulate what so far discussed in the following
\begin{prop} \label{reducibles}
Let $S$ be a singular cubic surface of type $A_1$, let $P$ be the singular point.
\begin{enumerate}[label=(\roman*)]
\item
The surface $S$ has 6 distinguished lines $\overline{E}_1,..,\overline{E}_6$, they are the only ones through the singular point. For each pair $(\overline{E}_{i},\overline{E}_{j})$ there exists a unique line $R_{ij}$ meeting both of them not in $P$.
\item \label{alpha-ij} Each ordered pair $(\overline{E}_{i},\overline{E}_{j})$ determines the family of twisted cubics of type $\alpha$ corresponding to the linear system $|E_i - E_j -  K_{\tilde{S}}|$.
\item Each triple $(\overline{E}_{i},\overline{E}_{j},\overline{E}_{k})$ determines the family of twisted cubics of type $\beta$ corresponding to the linear system $|H - E_i - E_j - E_k -  K_{\tilde{S}}|$.\footnote{Even though $H$ seems to depend on the choice of the resolution it does not because $H = - \Gamma - K_{\tilde{S}}$ where $\Gamma$ is the $(-2)$-curve.}
\end{enumerate}
\end{prop}

The general element of the divisor $D'$ is a family of generalised twisted cubic lying on a singular surface of type $A_1$ and can be of type $\alpha, \beta$ or $\gamma$.
\begin{defi}\label{def-divisors}
 We define $D'_\alpha$, respectively $D'_\beta$ and $D'_\gamma$, as the closure of the sets in $Z'$ of families of generalised twisted cubics lying on $A_1$-singular cubic surfaces of type $\alpha$, respectively $\beta$ and $\gamma$.
 We call $D_\alpha$, respectively $D_\beta$, the image $\pi(D'_\alpha)$ in $Z$, respectively $\pi(D'_\beta)$.
\end{defi}
\begin{prop}
The closed set $D'_\gamma$ is an irreducible divisor in $Z'$.
\end{prop}
\begin{proof}
The set $D'_\gamma$ coincides with the divisor of nCM generalised twisted cubics, which is irreducible \cite[Proposition~4.5]{LLSvS}.
\end{proof}
In the next sections we show the irreducibility of $D'_\alpha$, $D_\alpha$ (Corollary \ref{Dalpha-irr}) and of $D'_\beta$, $D_\beta$ (Corollary \ref{Dbeta-irr}).

\section*{The irreducible components of $D$}
\subsection*{The irreducible component $D_\alpha$}
Let $Y$ be a smooth cubic fourfold not containing a plane and let $F$ be its Fano variety of lines, which is an irreducible holomorphically symplectic variety of dimension 4.
Voisin \cite[Proposition~4.8]{Voisin} constructed a degree 6 rational map
\[
\varphi \colon F\x F \dashrightarrow Z.
\]
The construction goes as follows. Let $(l,l')$ be a general point in $F\x F$. It corresponds to a pair of non-coplanar lines $L,L'$. After having chosen a point $x$ on $L$, one takes the residual conic $Q_x$ to $L'$ of the intersection $Y \cap \langle x,L' \rangle$. The union of $Q_x$ and $L$ determines then the class of a generalised twisted cubic on the cubic surface $S:= \langle L, L' \rangle \cap Y$.
In other words, the lines $L$ and $L'$ determine the family $|\O_S (-K_{S}+ L - L')|$ of generalised twisted cubics on $S$.

The indeterminacy locus of $\varphi$ coincides with the variety $I$ of incident lines, which is irreducible of dimension 6, \cite[Theorem~1.2, Lemma~2.3]{Muratore}. The branch divisor of a resolution of $\varphi$ is a uniruled divisor, as remarked in \cite[Remark~4.10]{Voisin}, for details see \cite[Lemma~4.4]{Muratore}.
We consider the rational maps
\begin{align*}
\phi':= \pi^{-1} \circ \phi \colon & F \x F \dashrightarrow Z'   \\
\psi:= b \circ \pi^{-1} \circ \phi \colon & F \x F \dashrightarrow \Gr
\end{align*}
where $\pi: Z' \to Z$ is the morphism discussed above and $\pi^{-1}$ is its rational inverse map. Here $\Gr:=\Gr(V,4)$ parametrises 3-dimensional projective spaces in $\P^5 = \P(V)$.
As natural resolution of the indeterminacy locus of $\psi$ we consider the closure $\Gamma$ of its graph with projections $p\colon \Gamma \to F\x F$, $q\colon \Gamma\to \Gr$. We study its points by taking flat limits along curves $\lambda$ on $F\x F$ through points where the rational map is not defined.

\begin{lem}
About the fiber of the graph over $I$ we have:
\[
q(p^{-1}( I )) = D^{\Gr}.
\]
Moreover, for the general point $(l,l')$ in $I$ we have:
\[
p^{-1}(l,l') = \{ E \in \Gr: \langle L,L' \rangle \subset E \subset T_y Y \} \simeq \P^1
\]
 where $y$ is the unique intersection point of $L$ and $L'$.
\end{lem}
\begin{proof}
Let $i = (l,l')$ be a general point in $I$ and $\lambda\subset F\x F$ a smooth curve in an affine open intersecting $I$ exactly in $i$.
By the properness of the Grassmannian we get a morphism $\lambda \to \Gr$ and the curve parametrises a family $f\colon\mathscr{S}\to \lambda$ of cubic surfaces  contained in $Y$.
 Since singular cubic surfaces determine a divisor in the Grassmannian, we may assume that $\mathscr{S}_t$ is smooth for every $i\not =t \in \lambda$. By \cite[III, Theorem 10.2]{HAG} the smoothness of the morphism $f$ is equivalent to the smoothness of the surface $\mathscr{S}_i$.

We claim that the surface $\mathscr{S}_i$ is not smooth. Suppose to the contrary that the family were smooth, then the Picard groups $\Pic(\mathscr{S}_t) = H^2(\mathscr{S}_t, \bZ)$ would glue together in the local system $R^2 f_*\bZ$ on $\lambda$. Every point $t\in \gamma$ parametrises a pair of lines $(l_t,l'_t)$, which are disjoint for $t\not=i$. Taking their intersection product in $\Pic(\mathscr{S}_t)$ we then get
\[
0 = L_t\cdot L_t' = L_i \cdot L'_i = 1.
\] 
Hence, we conclude that $\mathscr{S}_i$ is singular.

 This shows the factorisation $p^{-1}(I) \to D^{\Gr} \subset \Gr$. Choosing $\lambda$ accurately, one proves that the morphism is dominant onto $D^{\Gr}$ and thus surjective.

Indeed, let $E\in D^{\Gr}$ be a $\P^3$ that contains two distinct incident lines $L$ and $L'$ meeting in a point $y$, in which the cubic surface $E\cap Y$ is singular. We consider the following diagram involving the tangent space $T_y Y$ of $Y$ at $y$, the normal bundle $N_{L|Y}$ of $L$ in $Y$ and its stalk $N_{L|Y}(y)$ at the point $y$ with the natural maps:
\[\xymatrix{
& T_yY\ar[d]\\
H^0(L,N_{L|Y})\ar[r]_{ev} & N_{L|Y}(y).
}
\]
The general line $L$ is of type I \cite[Definition~6.6]{griffiths}, that is, $N_{L|Y} \cong \O_{L}^{\oplus 2} \oplus \O_L(1)$, thus we may assume that the evaluation map $ev$ is surjective. 
Let $e$ be a vector in $T_y E$ not contained in $T_y \langle L, L' \rangle$. The image of $e$ under $T_y E \subset T_y Y \to N_{L|Y}(y)$ lifts to a vector $\bar{e} \in H^0(L,N_{L|Y})$, which corresponds to a deformation of $L$ and is represented by a curve $\lambda'$ in $F$ through $l$. If we set $\lambda = \lambda' \x \{l'\} \subset F\x F$, then the limit $\P^3$ computed along $\lambda$ coincides with $E$.

The first assertion is now proven.\bigskip

 The limit surface $\mathscr{S}_i$ is in effect singular in the intersection point $y$ of $L_i$ and $L_i'$. Indeed, we may assume $\mathscr{S}_i$ has one $A_1$-singularity in the point $P$. In virtue of \cite[\S 2]{riemenschneider}, after restricting to an analytic neighbourhood of $i$, there is a diagram
 \[\xymatrix{
T \ar[r]^r \ar[d] &\mathscr{S} \ar[d] \\
(\lambda',i') \ar[r]^f & (\lambda, i)
}
\]
where $f\colon(\lambda',i') \to (\lambda,i)$ is a finite Galois cover mapping $i'$ to $i$ and where $T\to \lambda'$ is a family of smooth surfaces such that $T_{t}\to \mathscr{S}_{f(t)}$ is an isomorphism for any $i'\not = t \in \lambda'$ and $T_{i'}\to \mathscr{S}_{i}$ is the minimal resolution of $\mathscr{S}_i$.
The surfaces $T_t$ are isomorphic to blow-ups of $\P^2$ in $6$ points, which are in general position for any $t\not =i'$, hence the groups $\Pic(T_t)$ form a local system over the all $\lambda'$. After further shrinking $\lambda'$ to a contractible neighbourhood of $i'$ the latter becomes trivial. 
The lines $L_{f(t)}\subset \mathscr{S}_{f(t)} \simeq T_t$ form a flat family over $\lambda \setminus \{i'\}$. Taking its closure we find a curve $X$ in $T_{i'}$, which completes the family to a flat family over all $\lambda'$ and corresponds to a section of the local system of Picard groups. The curve $X$ can be either the strict transform $\tilde{L}$ of the line $L$ or the union of $\tilde{L}$ and the $(-2)$-curve $\Gamma$, which arises as exceptional divisor of the resolution of $\mathscr{S}_i$.
Analogously, for the lines $L'_t$ we find the curve $X'$ in $T_{i'}$, which can be either $\tilde{L}'$ or the union of $\tilde{L}'$ and the curve $\Gamma$.
Since the intersection product of $X$ and $X'$ in $\Pic(T_{i'})$ must coincide with $L_{f(t)}\cdot L'_{f(t)} = 0$ for any $t \not = i'$, the only possibility is that $P$ lies on both the lines $L$ and $L'$.
\end{proof}

Let $\lambda$ be a curve as the one in the proof above, since $Z'$ is proper the Voisin map extends to a well-defined morphism
\[
\phi'_\lambda\colon\lambda \to Z'.
\]
We are interested in the image of $i$, which is represented by a family of generalised twisted cubic. By the previous lemma we know that any such curve lies on a singular surface.

\begin{lem}
For the general point $i\in I$ and the general curve $\lambda$ the limit twisted cubic $\phi'_\lambda(i)$ is of type $\alpha$.
\end{lem}
\begin{proof}
We may assume that the limit family of twisted cubics lies on a singular surface $S_i$ of type $A_1$ with one singularity at the point of intersection of $L,L'$.
The point $\phi'_\lambda (i)$ is represented by a family of generalised twisted cubics on $S_i$, which in turn corresponds to a linear system $A$ on the minimal resolution $\tilde{S}$ of $S_i$. In contrast, for any other point $t\not = i$ in $\lambda$ the image $\phi'_\lambda(t)$ consists of the family of curves in $|\O_{S_t} (-K_{S_t}+L_t - L_t')|$. 
 After passing to a Galois cover of an analytic neighbourhood of $i$ in $\lambda$ as before, we see that the linear system $A$ is equal to $|\O_{\tilde{S}} (-K_{\tilde{S}}+ \tilde{L} -\tilde{L}')|$,  where $\tilde{L}$ and $\tilde{L'}$ are the strict transforms of the two lines $L$ and $L'$.
Thus the limit family $\phi'_\lambda(i)$ is of type $\alpha$.
\end{proof}

In terms of the geometry of $Z$ we have thus proven the following.
\begin{cor}\label{Dalpha-irr}
The closed set $D'_\alpha$ is an irreducible uniruled divisor in $Z'$. Its image $D_\alpha$ in $Z$ coincides with the branch locus of a resolution of the Voisin map.
\end{cor}

\subsection*{The irreducible component $D_\beta$}
We consider the variety of triples of lines with non-trivial common intersection:
\[
I_3:= \{ (l_1,l_2,l_3) \in F\x F\x F: L_1\cap L_2\cap L_3 \not = \emptyset \}.
\]

\begin{lem}
The variety $I_3$ is irreducible of dimension $7$. 
\end{lem}
\begin{proof}
Let $\bL \subset F\x Y$ be the universal family of lines on $Y$ parametrised by $F$, its threefold product fits in the diagram
\[\xymatrix{
\bL \x \bL \x \bL \ar[r]^p\ar[d]^q & Y\x Y\x Y \\
F\x F\x F.
}
\]
where $p$ and $q$ denote the natural projections.
The variety $I_3$ is the image via $q$ of $J:=p^{-1} (\Delta)$, where $Y \simeq \Delta \subset Y\x Y\x Y$ is the diagonal embedding. Since $J$ is locally cut out by $8$ equations, all its irreducible components have dimension greater than or equal to $7$. The restriction of $q$ to $J$ is birational and just contracts the large diagonal 
\[
\{(l_1,l_2,l_3)\in I_3: l_{i_1} = l_{i_2} \textrm{ for some } i_1\not = i_2\}
\]
which has dimension 6. Thus all irreducible components of $I_3$ have dimension at least $7$.
Via the projection $F\x F\x F \to F\x F$ onto the first two factors  the variety $I_3$ is fibred over the irreducible variety $I$ of dimension $6$ \cite[Lemma~2.3]{Muratore}: 
\[
p_{12}: I_3 \to I.
\]
We study its fibres.
\begin{itemize}
\item if $(l,l')$ lies on the diagonal of $F\x F$, that is $L = L'$, then its preimage is the variety $F_L$ of lines intersecting $L$. 
\item In contrast, if $(l,l')\in I$ is a point such that $L\cap L' = \{ y \}$ then the fibre $p^{-1}_{12}(l,l')$ is the variety $C_y$ of lines through $y$.
\end{itemize}
The variety $C_y$ of lines through a given point $y$ is a curve, which is in general irreducible, except for finitely many points in $Y$ for which is a surface \cite[Proposition~2.4]{starr}. The variety $F_L$ admits a rational map to $L$ well-defined away from $l\in F_L$:
\[
F_L \dashrightarrow L,\ r \mapsto R \cap L.
\]
The fibre over a point $y\in L$ is the variety $C_y$, thus $F_L$ is a surface. It follows that $I_3$ is irreducible.
\end{proof}

A general triple $(l_1, l_2, l_3)$ in $I_3$ spans a $\P^3$ which intersects the cubic fourfold $Y$ in a singular cubic surface of type $A_1$: the singular point being the unique common intersection point of the three lines. According to our discussion in the previous section, this data determines the class of a generalised twisted cubic of type $\beta$ (cf. Proposition \ref{reducibles}). We have therefore constructed a rational map
\[
\rho \colon I_3 \dashrightarrow  Z'.
\]
which is dominant onto $D'_\beta$.
As immediate consequence we get\
\begin{cor}\label{Dbeta-irr}
The closed set $D'_\beta$ in $Z'$ as well as its image $D_\beta$ in $Z$ is an irreducible divisor.
\end{cor}

\bibliographystyle{plain}
\bibliography{divisor}

\end{document}